\documentclass[12pt,a4paper]{amsart}
\usepackage[T1]{fontenc}
\usepackage[utf8]{inputenc}
\usepackage{lmodern}
\usepackage{microtype}
\usepackage[a4paper,margin=27mm]{geometry}
\usepackage{amsmath,amssymb,mathtools}
\usepackage[hidelinks]{hyperref}
\usepackage[dvipsnames]{xcolor}            
\usepackage[normalem]{ulem}

\hypersetup{pdftitle={On the power of adaption and randomization 2: Non-symmetric convex sets},
 pdfauthor={David Krieg, Erich Novak, Mario Ullrich}}
\numberwithin{equation}{section}
\newtheorem{theorem}{Theorem}[section]
\newtheorem{lemma}[theorem]{Lemma}
\newtheorem{proposition}[theorem]{Proposition}

\theoremstyle{remark}
\newtheorem{remark}[theorem]{Remark}
\DeclareMathOperator{\vol}{vol}

\DeclareMathOperator{\aff}{aff}
\DeclareMathOperator{\conv}{conv}
\DeclareMathOperator{\spanop}{span}
\newcommand{\R}{\mathbb R}
\newcommand{\E}{\mathbb E}
\newcommand{\eada}{e^{\mathrm{ada}}}
\newcommand{\eran}{e^{\mathrm{ran}}}
\newcommand{\enon}{e^{\mathrm{non}}}
\newcommand{\one}{\mathbf 1}
\newcommand{\norm}[1]{\lVert#1\rVert}
\allowdisplaybreaks[1]

\title[On the power of adaption and randomization 2]{On the power of adaption and randomization 2:\\ Non-symmetric convex sets}
\author{David Krieg}
\author{Erich Novak}
\author{Mario Ullrich}
\date{October 24, 2026}

\begin{document}
\begin{abstract}
Let $S:X\to Y$ be a bounded linear operator between real Banach spaces,
and let $F\subset X$ be nonempty, bounded and convex. We prove
\[
 a_{2n-1}(S,F)\le 27n\min\{\eran_n(S,F),\eada_n(S,F)\},\qquad n\ge1.
\]
Here $a_k$ is the minimal worst-case error of continuous affine algorithms
of rank at most $k$, and $\eran_n$ and $\eada_n$ are the minimal errors of
adaptive randomized and deterministic algorithms using at most $n$
continuous linear measurements. 
This shows that the gain of adaption and randomization, and even
nonlinearity, is limited by a factor $27n$ for approximation in convex sets.
We discuss the optimality of the result by examples.
\end{abstract}
\maketitle

\section{Introduction and main result}\label{sec:intro-new}

Throughout, $X$ and $Y$ are real Banach spaces, $S\colon X\to Y$ is a
bounded linear operator, and $F\subset X$ is nonempty, bounded and
convex. We want to approximate $Sf$ for $f\in F$ in the norm of $Y$,
using at most $n$ values of continuous linear functionals of $f$, and
we compare the minimal worst-case errors of nonadaptive, adaptive and
randomized algorithms. For symmetric $F$, that is $F=-F$, these
relations are well understood; without symmetry, the known bounds are
weaker by a factor $\sqrt n$. We close this gap.

A \emph{deterministic algorithm using $n$ measurements} is a map
$A\colon X\to Y$ of the form
\begin{equation}\label{eq:algorithm}
 A(f)=\varphi(y_1,\ldots,y_n),\qquad
 y_i=\psi_i(y_1,\ldots,y_{i-1})(f),\quad i=1,\ldots,n,
\end{equation}
with arbitrary maps $\psi_i\colon\R^{i-1}\to X'$ and $\varphi\colon\R^n\to Y$,
where $X'$ is the dual space of $X$. The $i$th functional may depend on
the previous values; this is adaption. The algorithm is
\emph{nonadaptive} if all $\psi_i$ are constant, that is, $y_i=L_i(f)$
for fixed $L_1,\ldots,L_n\in X'$, and \emph{affine} if, in addition,
$\varphi$ is affine, that is, $A(f)=y_0+\sum_{i=1}^nL_i(f)\,v_i$ with
$y_0,v_i\in Y$. We write
\[
 \eada_n(S,F),\qquad\enon_n(S,F),\qquad a_n(S,F)
\]
for the infimum of $\sup_{f\in F}\norm{Sf-A(f)}_Y$ over all
deterministic, all nonadaptive and all affine algorithms using $n$
measurements, respectively; clearly $\eada_n\le\enon_n\le a_n$.
For the unit ball $F$ of $X$, the $a_n$ are the approximation numbers
of $S$; see~\cite{KNU2025,Mathe1990,TWW1988} for s-numbers in
information-based complexity.

A \emph{randomized algorithm using $n$ measurements} is a family
$(A_\omega)_{\omega\in\Omega}$ of deterministic algorithms
using $n$ measurements, indexed by a probability space
$(\Omega,\mathcal A,\mathbb P)$; the randomness is independent of the
input. We call it \emph{measurable} if, restricted to any
finite-dimensional affine subspace of $X$, the decision rules
$\psi_i^\omega$, the reconstruction $\varphi^\omega$ and the error
$\norm{Sf-A_\omega(f)}_Y$ are jointly measurable in $\omega$ and in
the observed values, respectively in $\omega$ and $f$. This is all
that is used in Section~\ref{sec:proof}. Then
\[
 \eran_n(S,F)=\inf\,\sup_{f\in F}\E\norm{Sf-A_\omega(f)}_Y
\]
is the infimum over all measurable randomized algorithms using $n$
measurements. Since $\eada_n$ is defined without any measurability,
the inequality $\eran_n\le\eada_n$ is not claimed and the deterministic
case is treated separately.

Let us recall what is known; see also~\cite{KNU2025,Novak1996} for earlier surveys. 
The deterministic adaption problem in this setting
was studied by Novak~\cite{Novak1995}
who used the Bernstein numbers $b_n$ as a main tool and proved $\enon_n \le C n^2 \eada_n$.  
For randomized algorithms, Kunsch~\cite{Kunsch2017} proved the lower bound $\eran_n \ge \frac1{30}b_{2n-1}$. 
Together these results show that 
$\eran_n$ and $\eada_n$ can gain at most a factor $Cn^2$ over $\enon_{2n-1}$ 
with a constant $C$. 
In particular, adaption and randomization can increase the order of convergence by no more than~2.
The first part of this work~\cite{KNU2025} reduced the exponent to $3/2$ 
by extending a bound between Gelfand numbers 
and Hilbert numbers~$h_n$ due to Pietsch~\cite{Pietsch-ideals} 
to non-symmetric sets, see also~\cite{U24}. 
The bound was further improved by Ullrich~\cite[Theorem~7]{Ullrich2026}, which gives
$a_{2n-1}\le C n^{3/2} \min \{\eran_n,\eada_n\}$ with a constant $C$,
and with $n$ in place of $n^{3/2}$ for Hilbert targets $Y$, or if $F$ is symmetric. 

We reduce this factor to order $n$,
and to order $\sqrt n$ for Hilbert targets.
To obtain these results, we require a new lower bound for randomized and adaptive algorithms. 
We cannot use $h_n$ 
since the factor $n^{3/2}$ in the comparison $a_n\le Cn^{3/2}h_n$ from \cite{Ullrich2026} 
is needed,
see~\cite{DU2026}. 
Whether we can use $b_n$ is unclear, 
since we do not know whether $a_n \le C n b_n$ holds for general convex sets $F$.
As a replacement, 
we compare  both $a_{2n-1}$ and the
minimal errors $\eran_n$ and $\eada_n$ with 
the affine determinants $D_k$
introduced below.
We obtain the following result. 

\begin{theorem}\label{thm:main}
Let $X$ and $Y$ be real Banach spaces, $S\colon X\to Y$ a bounded linear
operator, and $F\subset X$ nonempty and convex with $S(F)$ bounded.
For every $n\ge1$,
\begin{equation}\label{eq:main2}
 \enon_{2n-1}(S,F)\le a_{2n-1}(S,F)
 \le27n\min\{\eran_n(S,F),\eada_n(S,F)\}.
\end{equation}
If $Y$ is a Hilbert space, then
\begin{equation}\label{eq:main2-hilbert}
 a_{2n-1}(S,F)\le\frac{27}{2}\sqrt n\,
           \min\{\eran_n(S,F),\eada_n(S,F)\}.
\end{equation}
\end{theorem}

The deterministic part holds without any measurability, and the randomized part also for the
larger root mean square error. 
Both 
inequalities are sharp up to logarithmic factors, even for symmetric
$F$: for the identities $\ell_1^m\to\ell_\infty^m$ and
$\ell_1^m\to\ell_2^m$ with suitable $m=m(n)$ and $F$ the unit ball of $\ell_1^m$, the ratio
$a_{2n-1}/\eran_n$ is at least $cn/\log(n+1)$, respectively
$c\sqrt{n/\log(n+1)}$, with a constant $c$, see~\cite{KNW2024,KW2026}.  
The papers also give examples of a single problem where the order of convergence differs 
by $1$ and $1/2$, respectively. 
The optimal logarithmic factors in these comparisons, and whether the oversampling is needed, are open.

We add that our result is uniform over arbitrary convex input sets and
does not replace stronger lower bounds available for special 
spaces.  
In particular, if $X$ and~$Y$ are Hilbert spaces, and $F$ is the  unit ball of $X$,  
then the constant $27n$ can be replaced by 2, see~\cite{KNU2025,Novak1992}. 
Moreover, Bernstein-number estimates~\cite{Kunsch2017} and
Gaussian lower bounds~\cite{Heinrich1992}  due to Kunsch and
Heinrich
can yield sharper
problem-dependent results. 

\begin{remark}[Symmetric $F$] 
For symmetric classes $F$ and deterministic algorithms, it is a classical result 
of \cite{GM1980,TW80} that adaption can gain at most a constant factor, 
see also~\cite{Bakhvalov1971,CW2004}.  
Whether the same holds for randomized algorithms has only been settled recently.
The answer is no:
examples where adaptive randomized algorithms are much better than
nonadaptive randomized algorithms are given by Heinrich \cite{Heinrich2024a,He24b,Heinrich2024b} 
in the setting where only function evaluations are admissible measurements;
Kunsch, Novak and Wnuk \cite{KNW2024,KW2026} give examples in our setting of measurements from $X'$. 
\end{remark}

The proof of Theorem~\ref{thm:main} uses a geometric quantity: for $k\ge1$, let $D_k$ be the
supremum of $|\det[BS(f_1-f_0),\ldots,BS(f_k-f_0)]|$ over
$f_0,\ldots,f_k\in F$ and linear maps $B\colon Y\to\ell_2^k$ with
$\norm B\le1$. 
That is, we look for a $2n$-dimensional simplex in $F$ and a contraction such that the
	contracted image simplex has maximal volume. 
If $D_{2n}=0$, an affine
algorithm with $2n-1$ measurements is exact. 
Otherwise, we use the following steps: 
\begin{itemize}
\item We first show that $a_{2n-1}\le2\sqrt n\,(D_{2n}/D_n)^{1/n}$, without the factor
 $2\sqrt n$ for Hilbert spaces $Y$, see Section~\ref{sec:det}.
\item Then, for every $\varepsilon>0$, we fix a simplex $\Delta\subset F$ with $2n+1$
 vertices and a contraction $B$ with determinant at least
 $(1-\varepsilon)D_{2n}$, and let $\mu_\Delta$ be the uniform distribution on
 $\Delta$. We show in Section~\ref{sec:uniform} that, for every measurable deterministic~$A$ using $n$
 measurements,
 \begin{equation}\label{eq:outline-lower}
  \int_\Delta\norm{Sf-A(f)}_Y\,d\mu_\Delta(f)
  \ \ge\ \frac{2(1-\varepsilon)^{1/n}}{27\sqrt n}
     \left(\frac{D_{2n}}{D_n}\right)^{1/n}.
 \end{equation}
We use that, for any $y \in \R^n$, the set of inputs $f\in X$ that result in the measurement $y$ is an 
affine subspace $E_y$. The output $A(f)$ is constant on $E_y$ and the conditional 
distribution of $\mu_\Delta$ is uniform on the section $\Delta \cap E_y$. 
Proving \eqref{eq:outline-lower} therefore reduces to bounding from below the average distance 
of a uniformly distributed element from $S(\Delta \cap E_y)$ to a fixed point from $Y$ 
and then averaging over $y$, which has density proportional to the volume of $\Delta \cap E_y$.

\item For measurable randomized $(A_\omega)$,
 apply~\eqref{eq:outline-lower} to each $A_\omega$ and use Fubini:
 $\sup_f\E\norm{Sf-A_\omega(f)}_Y
 \ge\E\int_\Delta\norm{Sf-A_\omega(f)}_Y\,d\mu_\Delta(f)$.
 This is Bakhvalov's argument~\cite{Bakhvalov1959}; one input measure
 serves for all algorithms.
\item For deterministic $A$ without measurability, we consider only a central section, and prove that~\eqref{eq:outline-lower} holds for the
 worst case instead of the average.
\end{itemize}
We combine these ingredients, and prove Theorem~\ref{thm:main}, in Section~\ref{sec:proof}.
Remark~\ref{rem:sixteen} 
discusses a possible improvement of the constant 27 to 16 for randomized algorithms;
this requires a more complicated measure on a different simplex,
with Gaussian weights on the vertices in place of the uniform
distribution.

\bigskip

\section{Affine determinants and the choice of the simplex}
\label{sec:det}

For $k\ge1$, define
\begin{equation}\label{eq:D}
 D_k=D_k(S,F)=
 \sup_{\substack{f_0,\ldots,f_k\in F\\
                 B\in\mathcal L(Y,\ell_2^k),\ \norm B\le1}}
 \bigl|\det[BS(f_1-f_0),\ldots,BS(f_k-f_0)]\bigr|.
\end{equation}
Thus $B$ is a linear contraction defined on all of $Y$. The determinant
is $k!$ times the Euclidean volume of the image simplex. Hadamard's
inequality shows that $D_k<\infty$. Moreover, $D_k>0$ precisely when the dimension of 
the affine span 
$\aff S(F)$ is at least $k$. 

\begin{remark}[Relation to other quantities]\label{rem:Grothendieck}
Writing $By=(\ell_1(y),\ldots,\ell_k(y))$, the determinant
in~\eqref{eq:D} has entries $\ell_i(S(f_j-f_0))$, with the joint
constraint $\sum_i|\ell_i(y)|^2\le\norm y_Y^2$. 
Replacing this normalization by $\max_i|\ell_i(y)|\le\norm y_Y$, 
we obtain the Grothendieck determinant of $S$, see~\cite{PT89}, 
which has already been used in the context of sampling, 
see~\cite[p.~127]{Novak1995} or \cite[Section~5.3]{KU2026}. 
Also note that a quantity like $D_k$, but with ``$\ell_2$-normalization on both sides'', 
has recently been introduced in~\cite{Ullrich2026} to compare approximation and Hilbert numbers.
\end{remark}

\medskip

If $D_{2n}=0$, then the 
space  $\aff S(F)$ has dimension at
most $2n-1$. A bounded projection $P$ onto the respective linear subspace 
gives the exact
affine algorithm $f\mapsto Sf_0+P(Sf-Sf_0)$. 
The inequalities \eqref{eq:main2} and \eqref{eq:main2-hilbert} from 
Theorem~\ref{thm:main} are then immediate. We therefore assume from now on that $D_{2n}>0$.

\medskip

We begin with an upper bound on $a_n$ in terms of quotients of $D_k$, 
which is reminiscent of~\cite[Lemma~4]{Ullrich2026}.

\begin{lemma}\label{lem:interpolation}
If $D_{2n}>0$, then
\begin{equation}\label{eq:ratio}
 a_{2n-1}(S,F)\le2\sqrt n\left(\frac{D_{2n}}{D_n}\right)^{1/n}.
\end{equation}
For a Hilbert target, the factor $2\sqrt n$ may be omitted.
\end{lemma}
\begin{proof}
We first show, for $k\ge2$ and $D_{k-1}>0$, that
	\begin{equation}\label{eq:one-step}
		a_{k-1}D_{k-1}\le \theta_k D_k,\qquad
		\theta_k=\frac{k^{k/2}}{(k-1)^{(k-1)/2}}.
	\end{equation}
Choose vertices $f_0,\ldots,f_{k-1}$ and a contraction
$B:Y\to\ell_2^{k-1}$ such that, with $U : \ell_2^{k-1} \to X$,
$Uu=\sum_{j=1}^{k-1}u_j(f_j-f_0)$ and $M=BSU$,
$|\det M|\ge(1-\eta)D_{k-1}>0$. The affine interpolant and residual are
\[
I(f)=Sf_0+SUM^{-1}BS(f-f_0),\qquad r(f)=Sf-I(f).
\]
The linear part of $I$ has rank at most $k-1$, and $Br(f)=0$.
For $r(f)\ne0$, choose $\ell\in Y'$ with $\norm\ell=1$ and
$\ell(r(f))=\norm{r(f)}_Y$. The map
$\widetilde B=(\alpha B,\beta\ell)$ is a contraction whenever
$\alpha,\beta\ge0$ and $\alpha^2+\beta^2=1$.
For the vertices $f_0,\ldots,f_{k-1},f$ and the contraction $\widetilde B$,
	the matrix in~\eqref{eq:D} has the block form
	\[
	\begin{pmatrix}\alpha M & \alpha b\\ \beta c^{t} & \beta d\end{pmatrix},
	\qquad b=BS(f-f_0),\quad c^{t}u=\ell(SUu),\quad d=\ell(S(f-f_0)).
	\]
	By Schur's formula, its determinant is
	$\alpha^{k-1}\det M\cdot\beta\,(d-c^{t}M^{-1}b)$, and
	$d-c^{t}M^{-1}b=\ell\bigl(S(f-f_0)-SUM^{-1}BS(f-f_0)\bigr)
	=\ell(r(f))=\norm{r(f)}_Y$. Hence
	\[
	\alpha^{k-1}\beta|\det M|\norm{r(f)}_Y\le D_k.
	\]
Taking $\alpha^2=(k-1)/k$ and $\beta^2=1/k$, then the supremum over
$f$ and the limit $\eta\downarrow0$, proves~\eqref{eq:one-step}.
Now multiply \eqref{eq:one-step} for $k=n+1,\ldots,2n$ and use
$a_{2n-1}\le a_{k-1}$. The factors telescope:
\[
a_{2n-1}^n\le\left(\prod_{k=n+1}^{2n} \theta_k\right)\frac{D_{2n}}{D_n},
\qquad
\prod_{k=n+1}^{2n} \theta_k=\frac{(2n)^n}{n^{n/2}}=(2\sqrt n)^n.
\]
Taking $n$th roots proves the assertion.

\smallskip
\noindent\emph{Hilbert targets.} Let $Y$ be a Hilbert space. For
	$y_1,\ldots,y_k\in Y$ let
	\[
	V_k(y_1,\ldots,y_k)=\Big(\det\bigl[\langle y_i,y_j\rangle\bigr]_{i,j=1}^k\Big)^{1/2}
	\]
	be the $k$-dimensional volume of the parallelotope spanned by
	$y_1,\ldots,y_k$. We first show
	\begin{equation}\label{eq:D-volume}
		\sup_{B\in\mathcal L(Y,\ell_2^k),\ \norm B\le1}
		\bigl|\det[By_1,\ldots,By_k]\bigr| = V_k(y_1,\ldots,y_k).
	\end{equation}
	We may assume that the $y_j$ are linearly independent; otherwise both
	sides vanish. Let $E=\spanop\{y_1,\ldots,y_k\}$ and fix an orthonormal
	basis of $E$. If $c\in\R^{k\times k}$ is the coefficient matrix of
	$y_1,\ldots,y_k$ in this basis, then $|\det c|=V_k(y_1,\ldots,y_k)$.
	For a contraction $B$, the restriction $B|_E$ is a linear map
	between $k$-dimensional Euclidean spaces with all singular values at
	most one. Hence $|\det[By_1,\ldots,By_k]|=|\det B|_E|\,|\det c|\le V_k$.
	Conversely, equality holds for $B=JP_E$, where $P_E$ is the orthogonal
	projection onto $E$ and $J\colon E\to\ell_2^k$ is an isometry.
	By~\eqref{eq:D-volume},
	\[
	D_k=\sup_{f_0,\ldots,f_k\in F}V_k\bigl(S(f_1-f_0),\ldots,S(f_k-f_0)\bigr).
	\]
	
	Volumes satisfy the base-times-height formula: for all
	$y_1,\ldots,y_{k-1},y\in Y$,
	\begin{equation}\label{eq:base-height}
		V_k(y_1,\ldots,y_{k-1},y)
		=V_{k-1}(y_1,\ldots,y_{k-1})\,
		\operatorname{dist}\bigl(y,\spanop\{y_1,\ldots,y_{k-1}\}\bigr),
	\end{equation}
	as one sees by Gram--Schmidt orthogonalization of $y$ against the other
	vectors.
	
	Now let $0<\eta<1$ and choose $f_0,\ldots,f_{k-1}\in F$ with
	\[
	V_{k-1}\bigl(S(f_1-f_0),\ldots,S(f_{k-1}-f_0)\bigr)\ge(1-\eta)D_{k-1}>0.
	\]
	Let $P$ be the orthogonal projection onto the $(k-1)$-dimensional space
	$E'=\spanop\{S(f_j-f_0):1\le j\le k-1\}$, and put
	\[
	I(f)=Sf_0+PS(f-f_0).
	\]
	This is an affine algorithm whose linear part $PS$ has rank at most $k-1$ and
	$\norm{Sf-I(f)}_Y=\operatorname{dist}(S(f-f_0),E')$.
	For $f\in F$, the points $f_0,\ldots,f_{k-1},f$ are admissible
	in the supremum defining $D_k$. Hence~\eqref{eq:base-height} with
	$y=S(f-f_0)$ gives
	\[
	(1-\eta)D_{k-1}\,\norm{Sf-I(f)}_Y
	\le V_k\bigl(S(f_1-f_0),\ldots,S(f_{k-1}-f_0),S(f-f_0)\bigr)\le D_k.
	\]
	Taking the supremum over $f\in F$ and letting $\eta\downarrow0$ yields
	$a_{k-1}D_{k-1}\le D_k$, which is~\eqref{eq:one-step} with $\theta_k=1$
	and yields \eqref{eq:ratio} without $2\sqrt{n}$.
\end{proof}

\medskip

We now show that the quotients of the $D_k$ can be bounded by another quantity $G_n(T)$ that will also 
appear as a lower bound for the errors in Section~\ref{sec:uniform}. 
This is based on a \emph{large} simplex in $F$, measured by means of the corresponding determinants.

For this, fix $0<\varepsilon<1$ and choose 
vertices $f_0,\ldots,f_{2n}\in F$
and a contraction $B:Y\to\ell_2^{2n}$ such that
\begin{equation}\label{eq:choice}
 \delta:=\bigl|\det[BS(f_1-f_0),\ldots,BS(f_{2n}-f_0)]\bigr|
          \ge(1-\varepsilon)D_{2n}>0.
\end{equation}
Fix this choice and put
\[
 \Delta=\conv\{f_0,\ldots,f_{2n}\},\qquad
 N=2n+1,\qquad \bar f=\frac1N\sum_{j=0}^{2n}f_j.
\]
The vertices are affinely independent, and $\Delta\subset F$.
The simplex is chosen before any algorithm or random seed. 

For the analysis, set
\begin{equation}\label{eq:coordinates}
 H_0=\{z\in\R^N:\textstyle\sum_{j=0}^{2n}z_j=0\},\qquad
 Vz=\sum_{j=0}^{2n}z_jf_j,\qquad T=BSV.
\end{equation}
Equip $H_0$ and $\R^{2n}$ with their Euclidean inner products 
and consider $T$ as a mapping from $H_0$ to $\R^{2n}$. Write
$\sigma_1(T)\ge\cdots\ge\sigma_{2n}(T)>0$ for the singular values, and put
\begin{equation}\label{eq:G}
 G_n(T)=\left(\prod_{j=n+1}^{2n}\sigma_j(T)\right)^{1/n},
 \qquad Q_n=\binom{2n+1}{n+1}^{1/(2n)}.
\end{equation}

\begin{lemma}\label{lem:singular}
For the choice~\eqref{eq:choice},
\begin{equation}\label{eq:a-G}
 G_n(T)\ge\frac{(1-\varepsilon)^{1/n}}{Q_n}
                    \left(\frac{D_{2n}}{D_n}\right)^{1/n}.
\end{equation}
\end{lemma}
\begin{proof}
Fix a linear isometry $U_0\colon\R^{2n}\to H_0$, i.e.\ an orthonormal basis
	of $H_0$, and let $W\colon\R^{2n}\to H_0$ map the $j$th unit vector to
	$e_j-e_0$. Since $BS(f_j-f_0)=T(e_j-e_0)$, the matrix in~\eqref{eq:choice}
	is $TW=(TU_0)(U_0^*W)$. 
	The singular values of $TU_0$ are those of $T$, so
	$|\det(TU_0)|=\prod_{j=1}^{2n}\sigma_j(T)$. Moreover,
	$\det(U_0^*W)^2=\det(W^*W)$ is the Gram determinant of the vectors
	$e_j-e_0$, $1\le j\le2n$. Their Gram matrix is $I+\one\one^t$, whose
	determinant is $1+2n=N$. Consequently,
	\begin{equation}\label{eq:full-product}
		\delta=\sqrt N\prod_{j=1}^{2n}\sigma_j(T).
	\end{equation}
	Let $Q$ be the matrix with columns $BS(f_j-\bar f)$, $0\le j\le2n$, and
	let $\Pi$ be the orthogonal projection of $\R^N$ onto $H_0$. 
	Since $f_j-\bar f=V\Pi e_j$, we have $Q=T\Pi$. In particular $Q\one=0$, and
	$QQ^*=T\Pi\Pi^*T^*=TT^*$, because $\Pi^*\colon H_0\to\R^N$ is the
	inclusion and hence $\Pi\Pi^*$ is the identity on $H_0$.
	Let $P$ have the first $n$ left singular vectors of $T$ as its rows, and
	let $R$ be the $(n+1)\times N$ matrix with rows $\one^t$ and $PQ$. As
	$Q\one=0$ and $PQQ^*P^*=PTT^*P^*=\mathrm{diag}(\sigma_1(T)^2,\ldots,\sigma_n(T)^2)$,
	the Gram matrix $RR^*$ is block diagonal and
	$\det(RR^*)=N\prod_{j=1}^n\sigma_j(T)^2$.
	By Cauchy--Binet, $\det(RR^*)$ is the sum of the squares of the
	$\binom N{n+1}$ minors $\det R_J$, where $J=\{j_0,\ldots,j_n\}$ selects
	$n+1$ columns. In $R_J$, subtract the column $j_0$ from the others.
	The first row becomes $(1,0,\ldots,0)$, and $\bar f$ cancels, so
	\[
	\det R_J=\det\bigl[PBS(f_{j_i}-f_{j_0})\bigr]_{i=1}^n .
	\]
	Since $P$ has orthonormal rows, $PB\colon Y\to\ell_2^n$ is a contraction,
	and hence $|\det R_J|\le D_n$. Therefore
	\begin{equation}\label{eq:top-product}
		N\prod_{j=1}^n\sigma_j(T)^2\le\binom N{n+1}D_n^2.
	\end{equation}
	Dividing~\eqref{eq:full-product} by the square root of~\eqref{eq:top-product}
	gives $\delta\le\binom N{n+1}^{1/2}D_n\prod_{j=n+1}^{2n}\sigma_j(T)
	=Q_n^nD_nG_n(T)^n$. Now use~\eqref{eq:choice}.
\end{proof}

\section{A lower bound for uniform input}  \label{sec:uniform}

In this section, we work in finite-dimensional Euclidean spaces, with approximation errors measured 
in the Euclidean norm. We establish an average-case lower bound for uniformly distributed 
inputs on a convex body 
(a compact convex set with nonempty interior), together with a corresponding worst-case bound for 
arbitrary deterministic algorithms. 
These estimates will later be applied in the coefficient space of the simplex selected in Section~\ref{sec:det}.

\begin{lemma}\label{lem:density}
If $p$ is a bounded probability density with bounded support in $\R^n$ and
$\int\norm y_2^2p(y)\,dy\le r^2$, then
\begin{equation}\label{eq:density}
 \int_{\R^n}p(y)^{1+1/n}\,dy\ge\sqrt{\frac n{2\pi e r^2}}.
\end{equation}
\end{lemma}
\begin{proof}
Let $q(y)=(n/(2\pi r^2))^{n/2}\exp(-n\norm y_2^2/(2r^2))$.
Since $q>0$, the inequality $t\log t\ge t-1$ with $t=p(y)/q(y)$, multiplied
	by $q(y)$ and integrated over $y$, gives
	$\int p\log(p/q)\ge\int(p-q)=0$ (Gibbs' inequality).
Hence
\[
 \int p\log p\ge\int p\log q
 \ge-\frac n2\log\frac{2\pi e r^2}{n}.
\]
Jensen's inequality with respect to $p(y)\,dy$ now gives
\[
 \int p^{1+1/n}
 \ge\exp\left(\frac1n\int p\log p\right)
 \ge\sqrt{\frac n{2\pi e r^2}}.
\]
The integrals are finite under the stated boundedness assumptions,
with $0\log0=0$.
\end{proof}

The Gaussian density $q$ is only an auxiliary function in this estimate;
the input distribution below is uniform.
Put
\begin{equation}\label{eq:gamma}
 \omega_n=\vol_n\{y\in\R^n:\norm y_2\le1\}, \qquad	\gamma_n=\frac n{n+1}\omega_n^{-1/n}\sqrt{\frac n{2\pi e}}.
\end{equation}

\begin{proposition}\label{prop:uniform}
Let $K\subset\R^{2n}$ be a convex body
and $T:\R^{2n}\to\R^{2n}$ be linear. 
Every measurable adaptive algorithm $A$ using at most $n$ linear
measurements satisfies
\begin{equation}\label{eq:average-lower}
 \frac1M\int_K\norm{Tz-A(z)}_2\,dz
 \ge\gamma_n\frac{M^{1/n}}r G_n(T)
\end{equation}
with $G_n(T)$ as defined in \eqref{eq:G} and
	\begin{equation}\label{eq:M-r}
		M=\vol_{2n}(K),\qquad r^2=\frac1M\int_K\norm z_2^2\,dz.
	\end{equation}
If the center of mass of $K$ is zero, then every deterministic adaptive
algorithm, without measurability assumptions, satisfies
\begin{equation}\label{eq:det-lower}
 \sup_{z\in K}\norm{Tz-A(z)}_2
 \ge\gamma_n\frac{M^{1/n}}r G_n(T).
\end{equation}
\end{proposition}
\begin{proof}
We may assume $G_n(T)>0$, so $T$ is injective.
For every $n$-dimensional subspace $W$, the volume Jacobian 
$J(T|_W)=\prod_{j=1}^n\sigma_j(T|_W)$, where $T|_W$
	is regarded as a map from the Euclidean space $W$ to $\R^{2n}$,
satisfies by the min-max principle that
\begin{equation}\label{eq:jacobian}
 J(T|_W)\ge\prod_{j=n+1}^{2n}\sigma_j(T)=G_n(T)^n.
\end{equation}
Hence $\vol_n(T(E)) \ge G_n(T)^n \vol_n(E)$ for measurable $E$ in a translate of $W$.

We also use an elementary volume estimate. If $Z$ is uniform on a set
of volume $v$ in an $n$-dimensional affine Euclidean space, then for
every point $c$ in the ambient space,
\begin{equation}\label{eq:distance-volume}
 \E\norm{Z-c}_2
 \ge\int_0^{(v/\omega_n)^{1/n}}
          \left(1-\frac{\omega_nt^n}{v}\right)\,dt
 =\frac n{n+1}(v/\omega_n)^{1/n}
\end{equation}
since a ball of radius $t$ meets the affine space in volume at most
$\omega_nt^n$.

\smallskip
\noindent\emph{Measurable adaptive algorithms.}
Successive orthogonalization simulates the observations of $A$ by
\[
 y_i=\langle u_i(y_1,\ldots,y_{i-1}),z\rangle,\qquad i=1,\ldots,n,
\]
where the vectors $u_i$ are orthonormal along each information path.
Subtract the already observed components of each proposed functional
and normalize the remaining coefficient vector. If that vector is zero,
the answer is already known. Append unused orthogonal directions when
necessary; the reconstruction ignores their answers. A fixed pivot rule
in Gram--Schmidt makes these choices measurable.

For a complete answer vector $y$, put
\[
 a_y=\sum_{i=1}^n y_iu_i(y_1,\ldots,y_{i-1}),\qquad
 W_y=\spanop\{u_1,\ldots,u_n\}^{\perp},\qquad L_y=a_y+W_y.
\]
The inputs producing $y$ are exactly the points of $L_y$. Successive
orthogonal decompositions give, for every nonnegative measurable $h$,
\begin{equation}\label{eq:adaptive-Fubini}
 \int_{\R^{2n}}h(z)\,dz
 =\int_{\R^n}\int_{W_y}h(a_y+w)\,dw\,dy.
\end{equation}
At each step the previous answers are fixed before making the next
orthogonal change of coordinates. Thus~\eqref{eq:adaptive-Fubini} follows by
successive applications of ordinary Fubini; only measurability of the
decision rules is required.
Measurable orthonormal frames identify the spaces $W_y$ with $\R^n$.

Consequently, for uniform input on $K$, the answers have density
\begin{equation}\label{eq:answer-density}
 p(y)=v(y)/M,\qquad v(y)=\vol_n(K\cap L_y),
\end{equation}
and, for almost every $y$ with $v(y)>0$, the conditional input law is
uniform on $K\cap L_y$. Since $K$ is bounded, $p$ is bounded and has
bounded support. Also $\norm y_2^2\le\norm z_2^2$, so
$\int\norm y_2^2p(y)\,dy\le r^2$.
Write $A(z)=c_y$ for $z\in L_y$, where $c_y$ is the reconstruction for
the answers $y$. Applying~\eqref{eq:adaptive-Fubini} to
$h(z)=\one_K(z)\norm{Tz-A(z)}_2$ and using~\eqref{eq:answer-density}, we
get
\begin{align*}
\frac1M\int_K\norm{Tz-A(z)}_2\,dz
&=\frac1M\int_{\R^n}\int_{K\cap L_y}\norm{Tz-c_y}_2\,dz\,dy\\
&=\int_{\{v>0\}}p(y)\,\frac1{v(y)}\int_{K\cap L_y}\norm{Tz-c_y}_2\,dz\,dy .
\end{align*}
For fixed $y$, the inner average is $\E\norm{TZ-c_y}_2$ with $Z$ uniform
on $K\cap L_y$. 
Since $TZ$ is uniform on $T(K\cap L_y)$, a subset of the
$n$-dimensional affine space $T(L_y)$ which by~\eqref{eq:jacobian} has volume 
at least $G_n(T)^nv(y)$, inequality \eqref{eq:distance-volume} and then
Lemma~\ref{lem:density} give
\begin{align*}
	\frac1M\int_K\norm{Tz-A(z)}_2\,dz
	&\ge\frac n{n+1}\omega_n^{-1/n}G_n(T)
	\int p(y)v(y)^{1/n}\,dy\\
	&=\frac n{n+1}\omega_n^{-1/n}G_n(T)M^{1/n}
	\int p(y)^{1+1/n}\,dy\\
	&\ge\gamma_n\frac{M^{1/n}}r G_n(T).
\end{align*}

\smallskip
\noindent\emph{Arbitrary deterministic algorithms.}
Assume now that the center of mass of $K$ is zero. Fix the functionals
selected by $A$ at the input $0$, and choose an $n$-dimensional subspace
$W$ in their common kernel. All points of $K\cap W$ produce the same
answers as $0$, and therefore have the same reconstruction $c=A(0)$.
Let $Z$ be uniformly
distributed on $K\cap W$. Since $A(Z)=c$
and $TZ$ is uniformly distributed on $T(K\cap W)$, a set of
$n$-dimensional volume at least
$G_n(T)^n\vol_n(K\cap W)$, relation
\eqref{eq:distance-volume} gives
\begin{equation}\label{eq:central-error}
 \sup_{z\in K}\norm{Tz-A(z)}_2 
\ge\E\norm{TZ-c}_2
 \ge\frac n{n+1}\omega_n^{-1/n}G_n(T)
                    \vol_n(K\cap W)^{1/n}.
\end{equation}

For $y\in W^\perp$, let $K_y=\{w\in W:y+w\in K\}$,
$v(y)=\vol_n(K_y)$ and $p(y)=v(y)/M$.
The Brunn--Minkowski inequality implies that $v^{1/n}$ is concave on
the orthogonal projection of $K$ onto $W^\perp$: convexity gives
$(1-t)K_y+tK_{y'}\subset K_{(1-t)y+ty'}$, and Brunn--Minkowski applies
in $W$; see~\cite{Gardner2002}.
Since the center of mass of $K$ is zero, $\int yp(y)\,dy=0$.
Jensen's inequality now gives
\begin{equation}\label{eq:center-Jensen}
 v(0)^{1/n}\ge\int v(y)^{1/n}p(y)\,dy
       =M^{1/n}\int p(y)^{1+1/n}\,dy.
\end{equation}
Again $p$ is bounded with bounded support. 
Moreover, $p$ is the density
of $P_{W^\perp}z$ for $z$ uniformly distributed on $K$, so by Fubini
\[
\int\norm y_2^2p(y)\,dy=\frac1M\int_K\norm{P_{W^\perp}z}_2^2\,dz
\le\frac1M\int_K\norm z_2^2\,dz=r^2 .
\]
Substitution of
\eqref{eq:center-Jensen} and Lemma~\ref{lem:density} into
\eqref{eq:central-error} proves~\eqref{eq:det-lower}.
Only one information path was fixed, so no measurability of $A$ was used.
\end{proof}

\section{The uniform distribution on the simplex}
\label{sec:simplex-measure}

Return to the fixed simplex $\Delta$ chosen in~\eqref{eq:choice}, 
and consider the uniform distribution~$\mu_\Delta$ on it, which is
defined by 
\begin{equation}\label{eq:initial-measure}
\displaystyle
 \mu_\Delta(E)=\frac{\vol_\Delta(E\cap\Delta)}{\vol_\Delta(\Delta)},
\end{equation}
where $\vol_\Delta$ is any Lebesgue measure on the $2n$-dimensional affine space.

Every measurable subset has probability proportional to its
$2n$-dimensional volume. All Lebesgue measures on the same affine space
are proportional, so the normalized measure does not depend on any
choice of coordinates or Euclidean structure. Once the simplex has
been fixed, its uniform probability measure is fixed as well.

Equivalently, choose the barycentric coordinates jointly uniformly from
\begin{equation}\label{eq:barycentric}
 \Sigma_N=\{\lambda\in\R^N:\lambda_j\ge0,\ 
                         \textstyle\sum_{j=0}^{2n}\lambda_j=1\},
 \qquad f=\sum_{j=0}^{2n}\lambda_jf_j.
\end{equation}
The affine bijection in~\eqref{eq:barycentric} has constant volume
Jacobian, and hence sends normalized volume on $\Sigma_N$ to
$\mu_\Delta$. In particular, the measure is chosen independently of
the algorithm and its random seed.

To apply Proposition~\ref{prop:uniform}, center the coefficient simplex:
\begin{equation}\label{eq:coefficient-simplex}
 K=\Sigma_N-\frac1N\one\subset H_0,\qquad
 f=\bar f+Vz\quad(z\in K).
\end{equation}
The Euclidean structure on $H_0$ is only a convenient coordinate system
for the analysis of the already defined measure~\eqref{eq:initial-measure}.

\begin{lemma}\label{lem:simplex-moments}
For the body $K$ in~\eqref{eq:coefficient-simplex}, the center of mass is
zero and
\begin{equation}\label{eq:simplex-moments}
 M=\frac{\sqrt{2n+1}}{(2n)!},\qquad
 r^2=\frac{2n}{(2n+1)(2n+2)}
\end{equation}
with $M$ and $r$ from~\eqref{eq:M-r}.
\end{lemma}
\begin{proof}
The edge vectors $e_j-e_0$, $1\le j\le2n$, have Gram determinant $N$,
which proves the volume formula. Symmetry under coordinate permutations
gives $\E\lambda_j=1/N$. 
For fixed $j$ and $0\le t \le 1$, the sets $\{\lambda\in\Sigma_N:\lambda_j=t\}$ 
are copies of the $(N-2)$-dimensional face $\{\lambda_j=0\}$ scaled by $1-t$. 
By Fubini,
$\lambda_j$ therefore has density proportional to $(1-t)^{N-2}$, and hence equal to
$(N-1)(1-t)^{N-2}$. 
Integrating gives
$\E\lambda_j^2=2/[N(N+1)]$, and therefore
\[
 r^2=\E\sum_{j=0}^{2n}(\lambda_j-1/N)^2
     =\frac{2}{N+1}-\frac1N
     =\frac{N-1}{N(N+1)}.
\]
The center of mass assertion follows at the same time.
\end{proof}

Using $Q_n$ from~\eqref{eq:G} and $\gamma_n$ from~\eqref{eq:gamma}, put
\begin{equation}\label{eq:Cn}
 C_n=\frac{2Q_nr}{\sqrt n\,\gamma_nM^{1/n}}.
\end{equation}
Appendix~\ref{sec:constant} proves $C_n<27$ for every $n\ge1$ using
Stirling bounds and two direct small-dimension calculations.

\section{Proof of the main result}  \label{sec:proof}

The inequality $\enon_{2n-1}\le a_{2n-1}$ follows from the finite-rank
representation of affine algorithms. If $D_{2n}=0$,
the exact affine algorithm described in Section~\ref{sec:det} proves the theorem.
Otherwise fix $0<\varepsilon<1$, choose the simplex and contraction
by~\eqref{eq:choice}, and then choose its uniform measure
by~\eqref{eq:initial-measure}. Use $K,V,T,B$ from~\eqref{eq:coordinates}
and \eqref{eq:coefficient-simplex}.

For an algorithm $A$ on $F$, put
\[
 \widetilde A(z)=BA(\bar f+Vz)-BS\bar f.
\]
Each measurement of $\bar f+Vz$ is a known constant plus a linear
functional of $z$. Thus $\widetilde A$ uses at most $n$ linear
measurements. Since $B$ is a contraction on all of $Y$,
\begin{equation}\label{eq:transfer}
 \norm{Tz-\widetilde A(z)}_2
 \le\norm{S(\bar f+Vz)-A(\bar f+Vz)}_Y.
\end{equation}
For every measurable deterministic algorithm,
Proposition~\ref{prop:uniform} gives
\begin{equation}\label{eq:fixed-prior-lower}
 \int_\Delta\norm{Sf-A(f)}_Y\,d\mu_\Delta(f)
  \ge \frac1M \int_K \norm{Tz-\widetilde A(z)}_2\,dz 
 \ge\gamma_n\frac{M^{1/n}}r G_n(T).
\end{equation}
For a randomized algorithm $(A_\omega)$, apply~\eqref{eq:fixed-prior-lower} to each seed and
use Tonelli:
\begin{align*}
 \sup_{f\in F}\E_\omega\norm{Sf-A_\omega(f)}_Y
 \ge\E_\omega\int_\Delta\norm{Sf-A_\omega(f)}_Y\,d\mu_\Delta(f)
 \ge\gamma_n\frac{M^{1/n}}r G_n(T).
\end{align*}
Taking the infimum gives this lower bound for $\eran_n(S,F)$.

For an arbitrary deterministic algorithm, the center of mass of $K$ is
zero by Lemma~\ref{lem:simplex-moments}. Equations~\eqref{eq:det-lower}
and \eqref{eq:transfer} give the same lower bound for its worst-case
error. Consequently,
\begin{equation}\label{eq:both-lower}
 \min\{\eran_n(S,F),\eada_n(S,F)\}
 \ge\gamma_n\frac{M^{1/n}}r G_n(T).
\end{equation}
Combining Lemma 
\ref{lem:singular}
with~\eqref{eq:both-lower} and \eqref{eq:Cn} 
and letting $\varepsilon\downarrow0$,
we obtain
\begin{equation}\label{eq:main-lower}
 \min\{\eran_n(S,F),\eada_n(S,F)\}
 \ge \frac2{C_n\sqrt{n}} \left(\frac{D_{2n}}{D_n}\right)^{1/n}.
\end{equation}
Together with Lemma~\ref{lem:interpolation}, this gives
\[
 a_{2n-1}(S,F)\le
 C_n n 
       \min\{\eran_n(S,F),\eada_n(S,F)\}.
\]
For Hilbert targets, Lemma~\ref{lem:interpolation} omits $2\sqrt n$,
so the coefficient becomes $(C_n/2)\sqrt n$.
Finally, we use $C_n<27$.
This proves Theorem~\ref{thm:main}. 

\begin{remark}[A possible improvement of the constant]\label{rem:sixteen}
	It was suggested to the authors, in discussions with the AI systems ChatGPT
	and Claude, that for randomized algorithms the inequality
	\begin{equation}\label{eq:sixteen}
		a_{2n-1}(S,F)\le16n\,\eran_n(S,F),\qquad n\ge6,
	\end{equation}
	can be proven by a more involved analysis based on a different choice of
	the determinants and of the input measure. We did not verify the details.
	Instead of $D_k$, one considers
	\[
	\Gamma_k=\sup_{\substack{f_0,\ldots,f_k\in F\\
			\ell_1,\ldots,\ell_k\in Y',\ \norm{\ell_i}\le1}}
	\bigl|\det[\ell_i S(f_j-f_0)]_{i,j=1}^k\bigr|.
	\]
	The constraints on the rows are now individual: $(\ell_i)$ is a contraction
	into $\ell_\infty^k$. In particular, no factor $2\sqrt n$ appears in the
	analogue of Lemma~\ref{lem:interpolation}. Let $f_0,\ldots,f_{2n}\in F$ be
	vertices of an almost maximal configuration for $\Gamma_{2n}$, and let
	$\Delta_\Gamma$ be their convex hull. Instead of the uniform distribution
	on $\Delta_\Gamma$, one uses the following input distribution in
	Bakhvalov's argument. Let $u=3/2$, and let $\xi_0,\ldots,\xi_{2n}$ be
	independent with density
	\[
	q_u(x)=\frac{\phi(x-u)}{\Phi(u)}\one_{[0,\infty)}(x),\qquad
	\Xi=\sum_j\xi_j,\qquad f_\xi=\sum_j\frac{\xi_j}{\Xi}f_j,
	\]
	where $\phi$ and $\Phi$ are the standard normal density and distribution
	function. These are shifted Gaussian weights, conditioned to be
	nonnegative. With $m_u=u+\phi(u)/\Phi(u)=\E\xi_j$, the input measure is
	\begin{equation}\label{eq:weighted-Gaussian-prior}
		\nu(G)=\frac{\E\bigl[\Xi\,\one_{\{f_\xi\in G\}}\bigr]}{(2n+1)\,m_u},
		\qquad G\subset\Delta_\Gamma \text{ measurable}.
	\end{equation}
	It does not depend on the algorithm. The lower bound for the errors
	measured in $\ell_\infty^{2n}$ then uses Ball's bound for sections of the
	cube~\cite{Ball1989}.
\end{remark}

\appendix
\section{The constant 27}\label{sec:constant}

We verify the bound in~\eqref{eq:Cn} analytically. With $N=2n+1$ and
$\omega_n=\pi^{n/2}/\Gamma(n/2+1)$, substitution gives
\begin{equation}\label{eq:C-explicit}
 C_n=2Q_n\sqrt{\frac{2\pi e(n+1)}{n^3}}\,
       \frac{[\omega_n(2n)!]^{1/n}}{N^{1/2+1/(2n)}}.
\end{equation}
We use the standard Stirling bounds, valid for $x>0$,
\begin{equation}\label{eq:stirling}
 \sqrt{2\pi x}(x/e)^x\le\Gamma(x+1)
 \le\sqrt{2\pi x}(x/e)^xe^{1/(12x)};
\end{equation}
see~\cite[Section~5.11(ii)]{DLMF}. They give
\begin{equation}\label{eq:factorial-bound}
 [\omega_n(2n)!]^{1/n}
 \le\sqrt{\frac{2\pi e}{n}}\left(\frac{2n}{e}\right)^2
                     2^{1/n}e^{1/(24n^2)}.
\end{equation}

The elementary bound $\binom{2n}{n}\le4^n/\sqrt{3n+1}$ follows by
induction: it is an equality for $n=1$, and the induction step is
\[
 \frac{(2n+1)^2(3n+4)}{4(n+1)^2(3n+1)}\le1.
\]
Since $\binom{2n+1}{n+1}<2\binom{2n}{n}$, this implies
\begin{equation}\label{eq:Q-bound}
 Q_n\le2\left(\frac2{\sqrt{3n+1}}\right)^{1/(2n)}.
\end{equation}
Put $C_*=16\pi\sqrt2/e$. Substitution of~\eqref{eq:factorial-bound}
and \eqref{eq:Q-bound} into~\eqref{eq:C-explicit} gives
\[
 \frac{C_n}{C_*}\le
 \sqrt{\frac{2(n+1)}{2n+1}}\,
 \left(\frac8{(2n+1)\sqrt{3n+1}}\right)^{1/(2n)}e^{1/(24n^2)}.
\]
Using $\log(1+t)\le t$, we obtain
\[
 \log\frac{C_n}{C_*}\le\frac1n\left\{
 \frac14+\frac12\log\frac8{(2n+1)\sqrt{3n+1}}+\frac1{24n}\right\}.
\]
For $n\ge3$, we have
$(2n+1)\sqrt{3n+1}\ge7\sqrt{10}>22>8e$.
Thus the expression in braces is less than
$1/4-1/2+1/72<0$, and $C_n<C_*$. Moreover,
\[
 C_*<16\cdot\frac{22}{7}\cdot\frac{10}{7}\cdot\frac38
       =\frac{1320}{49}<27,
\]
using $\pi<22/7$, $\sqrt2<10/7$ and $e>8/3$.

It remains to consider $n=1,2$. Direct substitution in~\eqref{eq:C-explicit}
gives
\[
 C_1=16\sqrt{\frac{\pi e}{3}}<27,
 \qquad
 C_2^2=\frac{72\sqrt2}{5}\pi^2e<27^2.
\]
For the first inequality, the rational bounds $\pi<355/113$ and
$e<87/32$ give $C_1^2<82360/113<729$.
For the second, $\sqrt2<3/2$, $\pi<22/7$ and $e<11/4$ give
$C_2^2<143748/245<729$. This proves $C_n<27$ for every $n\ge1$.

\end{document}